\documentclass[11pt,a4paper]{amsart}
\usepackage{amsfonts}
\usepackage[top=35mm, bottom=35mm, left=32mm, right=32mm]{geometry}
\usepackage{mathptmx}
\usepackage{mathrsfs}
\usepackage{verbatim}
\usepackage{url}

\usepackage[all]{xy}
\usepackage[dvipsnames]{xcolor}
\usepackage[colorlinks=true,citecolor=blue]{hyperref}

\theoremstyle{plain}

\newtheorem{thm}{Theorem}[section]
\newtheorem{lem}[thm]{Lemma}
\newtheorem{prop}[thm]{Proposition}
\newtheorem{cor}[thm]{Corollary}

\theoremstyle{definition}
\newtheorem{defn}[thm]{Definition}
\newtheorem{exmp}[thm]{Example}
\newtheorem{rem}[thm]{Remark}
\newtheorem{ques}{Question}
\newcommand{\Z}{{\mathbb{Z}}}
\newcommand{\N}{\mathbb{N}}

\newcommand{\T}{\mathcal T}

\newcommand{\ep}{\varepsilon}

\DeclareMathOperator{\diam}{diam}

\begin{document}
\title{Mean equicontinuity, almost automorphy and regularity}
\author[F. Garc\'{\i}a-Ramos, T. J\"{a}ger and X. Ye]
{Felipe Garc\'{\i}a-Ramos, Tobias J\"{a}ger and Xiangdong Ye}
\date{\today }

\address[F. Garc\'{\i}a-Ramos]{CONACyT $\&$ Physics Institute of the Universidad Autonoma
de San Luis Potos, Av. Manuel Nava $\#$6, Zona Universitaria, C.P. 78290, San Luis
Potos, S.L.P., Mexico}
\email{fgramos@conacyt.mx}

\address[Tobias J\"{a}ger]{Institute of Mathematics, Friedrich Schiller University Jena, Germany}
\email{tobias.jaeger@uni-jena.de}

\address[X. Ye]
{Wu Wen-Tsun Key Laboratory of Mathematics, USTC, Chinese Academy of Sciences and Department of Mathematics, University of Science and
Technology of China, Hefei, Anhui 230026, China}
\email{yexd@ustc.edu.cn}

\begin{abstract}
The aim of this article is to obtain a better understanding and classification
of strictly ergodic topological dynamical systems with discrete spectrum. To
that end, we first determine when an isomorphic maximal equicontinuous factor
map of a minimal topological dynamical system has trivial (one point) fibres. In
other words, we characterize when minimal mean equicontinuous systems are almost
automorphic. Furthermore, we investigate another natural subclass of mean
equicontinuous systems, so-called diam-mean equicontinuous systems, and show
that a minimal system is diam-mean equicontinuous if and only if the maximal
equicontinuous factor is regular (the points with trivial fibres have full
Haar measure). Combined with previous results in the field, this provides a
natural characterization for every step of a natural hierarchy for strictly
ergodic topological models of ergodic systems with discrete spectrum. We also
construct an example of a transitive almost diam-mean equicontinuous system with
positive topological entropy, and we give a partial answer to a question of
Furstenberg related to multiple recurrence.

\end{abstract}
\maketitle


\section{Introduction}

The family of ergodic systems with discrete spectrum was one of the early
objects of study of formal Ergodic Theory (by Von Neumann in 1932 \cite{V}). It
is one of the few families where the isomorphism problem is well understood
(using spectral isomorphism): as stated by the Halmos-Von Neumann Theorem
\cite{HV}, every isomorphism class can be represented with a simple object, a
group rotation on a compact abelian group. More generally, the celebrated
Jewett-Krieger Theorem states that every ergodic system is isomorphic to a
strictly ergodic (uniquely ergodic and minimal) model.

However, even in the case of discrete spectrum, strictly ergodic systems may be
very different from a group rotation. Surprisingly, recent work has shown that
the family of \emph{all} topological dynamical systems that are strictly ergodic
models of discrete spectrum systems may exhibit a rich range of behaviours (from
a topological point of view) \cite{LTY, G, DG, FGJO, G06, HLSY}.  It turns out
that these properties can be classified in a natural hierarchy, in which the
best-understood systems with topological discrete spectrum -- namely the
equicontinuous systems -- just present the simplest subclass and only a small
fraction of the whole family. Apart from the intrinsic interest, the aim for a
better understanding of this class of systems is also motivated by mathematical
studies of quasicrystals, whose associated dynamical models often fall into this
category (that is, they combine strict ergodicity and discrete spectrum
\cite{BLM}).
\medskip

For simplicity, we restrict to the case of $\Z_+$-actions and say a {\it
  topological dynamical system} (t.d.s) is a pair $(X,T)$ consisting of a
compact metric space $X$ and a continuous transformation $T:X\to X$. When it
comes to classifying strictly ergodic systems (with or without discrete
spectrum), an important and well-known concept is that of the \emph{maximal
  equicontinuous factor} $(X_{eq},T_{eq})$ and its factor map $\pi_{eq}:X\to
X_{eq}$. The latter exists for every topological dynamical system
(e.g.\ \cite{A,D}) and can be obtained in a constructive way via the regionally
proximal equivalence relation \cite{A}. If $(X,T)$ is minimal, then
$(X_{eq},T_{eq})$ is both minimal and uniquely ergodic, and we denote its unique
invariant measure by $\nu_{eq}$.  The following is a natural hierarchy for
strictly ergodic systems with discrete spectrum, based on the invertibility
properties of the map $\pi_{eq}$:

\medskip $\pi_{eq}$ is a conjugacy (1-1)

$\Rightarrow\pi_{eq}$ is regular (almost surely 1-1, that is, 1-1 on a set of full
measure)

$\Rightarrow$ $\pi_{eq}$ is isomorphic and almost 1-1 (that is, 1-1 on a residual
subset)

$\Rightarrow\pi_{eq}$ is isomorphic

$\Rightarrow$ $(X,\mu,T)$ has discrete spectrum (where $\mu$ is the unique
invariant measure of $(X,T)$).
\medskip

Note that if an ergodic system has discrete spectrum, then it has to be
isomorphic to an equicontinuous system, but this isomorphism is not necessarily
given by the maximal equicontinuous factor map (actually every ergodic system
has a uniquely ergodic mixing topological model \cite{L87}). In other words, the
Kronecker factor in the Halmos-Von Neumann Theorem may be strictly bigger than
the maximal equicontinuous factor, and this is the case if and only if there
exist $L^2(\mu)$-eigenfunctions that are not continuous. Hence, the last
property is strictly weaker than the previous one.\medskip

An important notion in this context is that of \emph{mean equicontinuity} (or
mean-L-stability, a weakening of equicontinuity). It was introduced by Fomin
already in 1951 \cite{F}, but in the next 60 years only a few papers studying
this property appeared \cite{A59,S}. In particular, it was left as an open
question if minimal mean equicontinuous systems (equipped with their unique
ergodic measure) have discrete spectrum.  Recently, this question was answered
independently by Li, Tu and Ye \cite{LTY}, and by Garc\'{\i}a-Ramos \cite{G}
using different methods (also see \cite{HLTY,GM}). Garc\'{\i}a-Ramos
characterized when topological models of an ergodic system have discrete
spectrum, using a weaker notion called \emph{$\mu$-mean equicontinuity}. Li, Tu
and Ye proved that mean equicontinuity is stronger than just discrete spectrum
because the isomorphism to the group rotation can be achieved using the maximal
equicontinuous factor. Downarowicz and Glasner \cite{DG} proved the converse,
that is, if the maximal equicontinuous factor of a minimal systems yields an
isomorphism then the system must be mean equicontinuous. Furthermore they showed
that some minimal mean equicontinuous systems are not almost automorphic (
$\pi_{eq}$ is not almost 1-1). Altogether, this means that the last two steps in
the above hierarchy can be characterized using ($\mu$-)mean equicontinuity.

In this paper, we thus aim to characterize the second and third step (the first
is just equicontinuity). To that end, we introduce the notion of {\em frequent
  stability} and {\em diam mean equicontinuity}, closely related to Lyapunov
stable sets and mean equicontinuity. Under the assumption of minimality, we show
that for mean equicontinous systems the maximal equicontinuous factor map
$\pi_{eq}$ is almost 1-1 if and only if $(X,T)$ is frequently stable (Theorem
\ref{maininsection3}).  Moreover, we show
 that $\pi_{eq}$ is regular if and only if $(X,T)$ is diam-mean equicontinuous
 (Theorem \ref{mainresult4}).

 Other families of systems with
 discrete spectrum are null systems (zero topological sequence entropy
 \cite{kush,goodman}) and tame systems (Glasner \cite{G06}). It is now
 understood that these properties are very closely related to the hierarchy. It
 is not difficult to see that every equicontinuous system is null. Kerr and Li
 proved that these notions can be characterized using combinatorial independence
 and that every null system is tame \cite{KL05}. Fuhrmann, Glasner, J\"ager and
 Oertel showed that for every minimal tame system $\pi_{eq}$ is regular
 \cite{FGJO}. See \cite{HLSY,G18,G} for previous weaker results.
\smallskip

In summary, we obtain the following hierarchy for minimal systems:\medskip

Equicontinuity ($\Leftrightarrow$ topological discrete spectrum)

$\Rightarrow$null ($\Leftrightarrow$no unbounded independence)

$\Rightarrow$tame ($\Leftrightarrow$ no infinite independence)

$\Rightarrow$diam-mean equicontinuous ($\Leftrightarrow$ $\pi_{eq}$  regular)

$\Rightarrow$mean equicontinous and frequently stable ($\Leftrightarrow$ $\pi_{eq}$  almost 1-1 and isomorphic)

$\Rightarrow$mean equicontinuous ($\Leftrightarrow$ $\pi_{eq}$ isomorphic)

$\Rightarrow$ $\mu-$mean equicontinuous. ($\Leftrightarrow$ discrete spectrum)\medskip

Note that there exist counter-examples showing every implication is strict
\cite[Section~5]{goodman}, \cite[Section 11]{KL07} or \cite{FK}, \cite[Remark
  5.8]{G18}, \cite[Example 5.1]{DK}, and \cite[Theorem 3.1]{DG}.\smallskip

 We should mention that some work has been done for non-minimal systems and
 non-ergodic measures \cite{FGL,QZ,HLTY}, but further work will be required to
 extend our results in this direction. In contrast to this, with the use of
 \cite{FGL}, the extension to locally compact amenable group actions should be
 straightforward (except for Section 6).

In addition to the above-mentioned results, we also study transitive almost
diam-mean equicontinuity, a weakening of diam-mean equicontinuity. We show that
transitive almost diam-mean equicontinuous systems may have positive topological
entropy (note that systems with discrete spectrum have zero entropy). So, at
least in this sense, almost diam-mean equicontinuous systems do not appear to be
that different to almost mean equicontinuous systems (in contrast to the minimal
case). Finally we give a partial answer to a question of Furstenberg related to
multiple topological recurrence.
\medskip

The paper is organized as follows. In Section 2, we give some basic notions in
t.d.s. In Section 3 we show that for mean equicontinous systems $\pi_{eq}$ is
almost 1-1 if and only if $(X,T)$ is frequently stable.  In Section 4 we study
the basic properties of diam-mean equicontinuous systems and we prove that
$\pi_{eq}$ is regular if and only if $(X,T)$ is diam-mean equicontinuous. In
Section 5 we consider diam-mean sensitivity and almost diam-mean equicontinuity.
Furstenberg asked if for every t.d.s. $(X,T)$ and $d\in \N$, there is $x\in X$
such that $(x,x,\ldots,x)$ is a minimal point for $T\times T^2\times \ldots
\times T^d$. We give a positive answer for the class of mean equicontinuous
systems in Section 6.

\medskip
\noindent {\bf Acknowledgments.} We thank Eli Glasner for useful discussion in
the early stage of the paper, and Gabriel Fuhrmann for pointing out a way to
simplify the proof of Theorem~\ref{t.frequent_stability} and further helpful
remarks. X. Ye was supported by NNSF of China (11431012), T.~J\"ager by a
Heisenberg grant of the German Research Council (DFG-grant OE 538/6-1) and
F. Garc\'{\i}a-Ramos by CONACyT (287764).

\section{Preliminaries}
Throughout this paper, we denote by ${\mathbb{Z}}_{+}$ and $\mathbb{N}$ the
sets of non-negative integers and natural numbers respectively.

\subsection{Subsets of ${\mathbb{Z}}_{+}$ }

Let $F$ be a subset of ${\mathbb{Z}}_{+}$. The \textbf{upper density} and
\textbf{upper Banach density} of $F$ are defined by
\[
\overline{D}(F)=\lim\sup_{n\rightarrow\infty}\frac{\#\{F\cap\lbrack
	0,n-1]\}}{n}%
\]
and
\[
BD^{\ast}(F)=\limsup_{N-M\rightarrow\infty}\frac{\#\{F\cap\lbrack
	M,N-1]\}}{N-M}=\limsup_{n\rightarrow\infty}\left\{  \sup_{N-M=n}%
\frac{\#\{F\cap\lbrack M,N-1]\}}{n}\right\}  ,
\]
respectively, where $\#\{\cdot\}$ denotes the cardinality of the set. It is
clear that $\overline{D}(F)\leq BD^{\ast}(F)$ for any $F\subset{\mathbb{Z}%
}_{+}$.

\subsection{Topological dynamics}

\bigskip We say $(X,T)$ is a \textbf{topological dynamical system (t.d.s.)} if
$X$ is a compact metric space (with metric $d$) and $T:X\rightarrow X$ is a
continuous function. We denote the forward orbit of $x\in X$ by
$orb(x,T)=\{x,Tx,\ldots\}$ and its orbit closure by $\overline{orb(x,T)}$.  We
say $x\in X$ is \textbf{transitive} if $\overline{orb(x,T)}=X$. On the other
hand a t.d.s. $(X,T)$ is \textbf{transitive }if for any non-empty open sets
$U,V\subset X$ there exists $n\in\mathbb{N}$ such that $T^{n}U\cap
V\neq\emptyset$ and is \textbf{minimal }if every point of $X$ is transitive.  We
call $x\in X$ an \textbf{equicontinuity point} if for any $\varepsilon>0$ there
is a $\delta>0$ such that for every $y\in X$ with $d(x,y)<\delta$, we have
$d(T^{n}x,T^{n}y)<\varepsilon$ for all $n\in{\mathbb{Z}}_{+}$. A t.d.s.  is
\textbf{equicontinuous} if every $x\in X$ is an equicontinuity point (note that
by compactness every equicontinuous t.d.s. is uniformly equicontinuous, so a
t.d.s. is equicontinuous if and only if the family $\{T^n\}$ is equicontinuous).

Let $(X,T)$ and $(X^{\prime},T^{\prime})$ be two t.d.s. We say $(X^{\prime
},T^{\prime})$ is a \textbf{factor} of $(X,T)$ if there exists a surjective
continuous map $f:X\rightarrow X^{\prime}$ such that $f\circ T=T^{\prime}\circ
f$ (we refer to $f$ as the factor map).  A factor map $f:X\rightarrow
X^{\prime}$ is \textbf{almost 1-1} if $\{y\in X^{\prime}:\#(f^{-1}(y)=1) \}$ is
residual. If $(X^{\prime},T^{\prime})$ is minimal then $f$ is almost 1-1 if and
only if $\{y\in X^{\prime}:\#(f^{-1}(y)=1) \}$ is non-empty.  Every
t.d.s. $(X,T)$ has a unique (up to conjugacy) \textbf{maximal equicontinuous
  factor (m.e.f.)}  $(X_{eq},T_{eq})$ \cite{EG}, that is, an equicontinuous
factor such that every other equicontinuous factor of $(X,T)$ is a factor of
$(X_{eq},T_{eq})$. We will denote the maximal equicontinuous factor map by
$\pi_{eq}$. Every transitive equicontinuous system is strictly ergodic and we
denote this measure with $\nu_{eq}$.

Given a t.d.s. the \textbf{Besicovitch pseudometric} is given by
\[
\rho_b(x,y):=\lim\sup_{n\rightarrow\infty}\frac{1}{n}\sum_{i=1}^{n}%
d(T^{i}x,T^{i}y)\text{.}%
\]

\section{Characterization of almost automorphic mean equicontinuous systems}

In this section we give a characterization of almost automorphic mean equicontinuous systems.
We begin with the following definition.
\begin{defn}
	A t.d.s. $(X,T)$ is \textbf{mean equicontinuous} if for every $\varepsilon>0$ there exists
	$\delta>0$ such that if $d(x,y)\leq\delta$ then $\rho_b(x,y)\leq\varepsilon$.
\end{defn}

Actually a t.d.s $(X,T)$ is mean equicontinuous if and only if
for every $\varepsilon>0$ there exist $\delta>0$ and $N\in\N$ such that if $d(x,y)\leq\delta$ then $\frac{1}{n}\sum_{i=1}^{n}
d(T^{i}x,T^{i}y)<\varepsilon$ for any $n\ge N$ \cite{QZ}.

 $B_{\delta}(x)$ represents the open ball of radius $\delta$ centred at $x$. For $A\subset X$, we denote the diameter of $A$ by $\diam (A)$.  The following definition is
 crucial for the main result of the section.
\begin{defn}
Let $(X,T)$ be a t.d.s. We say that $x\in X$ is a \textbf{frequently stable} point of
$(X,T)$, if for every $\varepsilon>0$ there exists $\delta>0$ such that
\begin{equation*}
\overline{D}\left\{  i\in\mathbb{Z}_{+}:\diam(T^{i}B_{\delta}(x))>\varepsilon
\right\}  <1.
\end{equation*}
\end{defn}

We will need the following lemmas.
\begin{lem}[\cite{LTY}, \cite{G}]\label{LTY-G}
If $(X,T)$ is mean equicontinuous then
$\rho_b(x,y)=0$ if and only if $\pi_{eq}(x)=\pi_{eq}(y).$
\end{lem}

\begin{lem}[\cite{DG}]\label{lemmadb}
	\label{lemmadg}Let $(X,T)$ and $(Y,S)$ be minimal t.d.s. and $
	\pi:X\rightarrow Y$ a factor map. Then either $\pi$ is almost 1-1 or there
	exists $\varepsilon >0$ such that $\diam(\pi^{-1}(y))>\varepsilon $ for
	every $y\in Y$.
\end{lem}


\begin{lem}\label{known} Let $\pi:(X,T)\rightarrow (Y,S)$ be a factor map between two minimal t.d.s.
Then for each $\delta>0$ there is $\eta>0$ such that the
image under $\pi$ of any $\delta$-ball in $X$ contains an $\eta$-ball in $Y$.
\end{lem}
\begin{proof}Fix $\delta>0$ and $x\in X$. By minimality of $(X,T)$ there exists
$n\in\mathbb{N}$ such that $\bigcup_{i=0}^{n}T^{i}(B_{\delta}(x))$ covers $X$.
Therefore $\pi(B_{\delta}(x))$ needs to have non-empty interior, since
otherwise $Y=\bigcup_{i=0}^{n}\pi(T^{i}(B_{\delta}(x))$ would be a finite
union of meager sets, contradicting Baire's Theorem. Using compactness again,
one obtains that there exists $\eta>0$ such that the
image under $\pi$ of any $\delta$-ball in $X$ contains an $\eta$-ball in $Y$.
\end{proof}

Now we are ready to show
\begin{thm}\label{maininsection3}
\label{t.frequent_stability} Let $(X,T)$ be a minimal mean equicontinuous
t.d.s. Then the following are equivalent.

\begin{itemize}
\item[(i)] $\pi_{eq}:(X,T)\rightarrow(X_{eq},T_{eq})$ is almost 1-1;

\item[(ii)] every $x\in X$ is frequently stable;

\item[(iii)] there exists at least one frequently stable point of $x;$

\item[(iv)] $\pi_{eq}:(X,T)\rightarrow(X_{eq},T_{eq})$ is almost finite to 1.
\end{itemize}
\end{thm}

\begin{proof}
The fact that (ii) implies (iii), and (i) implies (iv) are obvious.\smallskip

First we will prove (iv) implies (i). Suppose that there exists $y\in X_{eq}$
such that $\pi_{eq} ^{-1}(y)$ is finite. Using Lemma \ref{LTY-G} we have that
for every $\varepsilon >0$ there exists $n\in \mathbb{N}$ such that
$d(T^{n}x,T^{n}x^{\prime })\leq \varepsilon $ for every $x,x^{\prime }\in
\pi_{eq} ^{-1}(y)$. Note that $T^nx, T^nx'\in \pi_{eq}^{-1}(T^ny)$. Using
minimality we have that $T:X\rightarrow X$ is surjective and hence for each
$z\in\pi_{eq}^{-1}(T^n_{eq}y)$ there is $x\in \pi_{eq}^{-1}(y)$ such that
$T^n(x)=z$. Hence $d(z,z')<\varepsilon$ for all $z,z'\in \pi_{eq}^{-1}(T^n_{eq}y)$.
Using Lemma~ \ref{lemmadg} we conclude that $\pi_{eq} $ is almost~1-1.

\medskip

Now we will prove (i) implies (ii).  By hypothesis there exists $y_0$ such that
$|\pi_{eq}^{-1}(y_0)|=1$. Let $\varepsilon>0$. There exists $\eta>0$ such that
$\diam(\pi_{eq}^{-1}(B_\eta(y)))\leq \varepsilon$ for every $y\in B_\eta(y_0)$.
Since $\nu_{eq}$ is fully supported, $B_\eta(y_0)$ has positive measure (and
contains a smaller ball with positive measure and boundary with null
measure). Thus by strict ergodicity, we have that
\[
\overline{D}(\left\{  n\in\mathbb{N}:T_{eq}^{n}(y)\in B_\eta(y_0)\right\})  >0
\]
for every $y\in X_{eq}$.

Let $x\in X$. Since $\pi_{eq}$ is continuous there exists $\delta>0$ such that
$\diam(\pi_{eq}(B_\delta(x)))<\eta$.  Without loss of generality we may assume
$T_{eq}$ is an isometry. This implies that $B_\eta(T^n_{eq}(\pi_{eq}
x))=T^n_{eq}B_\eta(\pi_{eq} x)$ for every $n\in\N$.  Consequently \[
T^n(B_\delta(x))\subseteq \pi_{eq}^{-1}(B_\eta(T^n_{eq}(\pi_{eq} x)))
\]
for all
$n\in\N$. So, if $T^n_{eq}(\pi_{eq} x)\in B_\eta(y_0)$ then $\diam(T^n(B_\delta(x))< \varepsilon$;
since this happens with positive density we conclude $(X,T)$ is frequently stable.

\medskip

Hence, it remains to show that (iii) implies (i). Without loss of generality
suppose that $\diam(X)=1$. We assume that $\pi_{eq}$ is not almost 1-1 and we
will show that no point in $X$ is frequently stable.  As $(X,T)$ is minimal and
mean equicontinuous, \cite{LTY,DG} yields that $\pi_{eq}$ is a measure-theoretic
isomorphism between $(X,T,\mu)$ and $(X_{eq},T_{eq},\nu_{eq})$, where $\mu$ and
$\nu_{eq}$ are the unique measures on $(X,T)$ and $(X_{eq},T_{eq})$
respectively.  Hence, there exists $X_{0}\subset X$, $Y_0\subset X_{eq}$ such
that $\mu(X_{0})=1$ and $\pi_{eq}:X_{0}\rightarrow Y_{0}=\pi_{eq}(X_{0})$ is
bijective. Let $\rho:Y_{0}\rightarrow X_{0}$ be the inverse of $\pi_{eq}$. This
implies that $\mu(A)=\nu_{eq}(\{y\in Y\mid\rho(y)\in A\}$ for every measurable
set $A\subseteq X$.

Since $\pi_{eq}$ is not almost 1-1, we have that $\left\vert \pi_{eq}
^{-1}(y)\right\vert >1$ for all $y\in X_{eq}$. By Lemma \ref{lemmadb}, there
exists $\varepsilon>0$ such that $\diam(\pi_{eq}^{-1}(y))>\varepsilon$ for every
$y\in X_{eq}$.\\

\medskip
\noindent {\bf Claim:} for any $\eta>0$ there exists $\kappa=\kappa(\eta)>0$ such that for
any $y\in {\color{blue}Y}$ and any $x\in\pi_{eq}^{-1}(y)$ we have that $\mu(B_{\varepsilon
/8}(x)\cap\pi_{eq}^{-1}(B_{\eta}(y))\geq\kappa$.

\medskip
Suppose for a contradiction that this is not the case. Then there exist
$y_{n}\in Y_0$ and $x_{n}\in\pi_{eq} ^{-1}(B_{\eta}(y_{n}))$ such that
$\lim_{n\rightarrow\infty}\mu(B_{\varepsilon
  /8}(x_{n})\cap\pi_{eq}^{-1}(B_{\eta}(y_{n}))=0$. By minimality, the topological
support of $\mu$ is all of $X$, so that every non-empty open subset of $X$ has
positive measure. In particular, if $y=\lim_{n\rightarrow \infty}y_{n}$ and
$x=\lim_{n\rightarrow\infty}x_{n}$ (where we go over to convergent subsequences
if necessary to ensure existence of the limits), the open set
$U=B_{\varepsilon/16}(x)\cap\pi_{eq}^{-1}(B_{\eta/2}(y))$ has positive measure. However,
for $n$ large enough $U$ is contained in $B_{\varepsilon
  /8}(x_{n})\cap\pi_{eq}^{-1}(B_{\eta}(y_{n}))$, leading to a contradiction. This
proves the claim.
\medskip

If we apply the above claim to any $y\in Y_0$ and two points $x_{1},x_{2}\in
\pi_{eq}^{-1}(y)$ with $d(x_{1},x_{2})>\varepsilon/2$ and define $A_{j}=\pi
\left(B_{\varepsilon/8}(x_{j})\cap\pi_{eq}^{-1}(B_{\eta}(y)\cap Y_0)\right)$ for
$j\in\left\{ 1,2\right\} $, then we obtain the following statement.%

\begin{equation}\tag{$*$} \label{c.frequent_stability}
	\begin{split} & \textit{ For every $\eta >0$ there exists $\kappa >0$ such that for all $y\in
			Y_{0}$ there are $A_1(y),A_2(y)\subset B_{\eta}(y)$}\\ &
          \textit{ such that $\nu_{eq}(A_1(y)),\nu_{eq}(A_2(y))>\kappa$ and
            $d(\rho (A_{1}(y)),\rho (A_{2}(y))) >\varepsilon /4$.}
	\end{split}
\end{equation}

Let $x\in X$ and $\delta>0$. Due to Lemma \ref{known} there exist $\eta>0$ and
$y\in \pi_{eq}(B_{\delta}(x))$ such that
$B_{\eta}(y)\subseteq\pi_{eq}(B_{\delta}(x))$.  Now, let
$\varphi:B_{\eta}(y)\rightarrow B_{\delta}(x)$ be a measurable mapping that
satisfies $\pi_{eq}\circ\varphi(y^{\prime})=y^{\prime}$ for all $y^{\prime}\in
B_{\eta}(y)$ (such a $\varphi$ exists by the well known result of Jankov-von Neumann). Then Lemma~\ref{LTY-G} implies that
\[
\ \lim_{n\rightarrow\infty}\frac{1}{n}\sum_{i=0}^{n}d(T^{i}\varphi(y^{\prime}),
T^{i}\rho(y^{\prime})) \ = \ 0
\]
for all $y^\prime\in B_{\eta}(y)$. Hence, dominated convergence yields
\[\
  \lim_{n\rightarrow\infty}\frac{1}{n}\sum_{i=0}^{n}\int_{B_{\eta}(y)}d(T^{i}\varphi(y^{\prime}),
  T^{i}\rho(y^{\prime}))\ d\nu_{eq}(y^{\prime})\ =\ 0\ ,
\]
which further implies
\[
\overline{D}\left\{  i\in\mathbb{Z}_{+}\mid\nu_{eq}\{\ y^{\prime}\in B_{\eta
}(y)\cap Y_0:d(T^{i}\varphi(y^{\prime}),T^{i}\rho(y^{\prime}))>\varepsilon/8\}\geq
\kappa/2\right\}  \ =\ 0\ .
\]
Consequently, we obtain that
\[
\overline{D}\left\{ i\in\mathbb{Z}_{+}\mid\nu_{eq}\{\ y^{\prime}\in B_{\eta }(y)\cap
Y_0:d(T^{i}\varphi(y^{\prime}),T^{i}\rho(y^{\prime}))\leq\varepsilon/8\}>
\nu_{eq}(B_{\eta}(y))-\kappa/2\right\} \ = \ 1.
\]
For any $i\in\mathbb{Z}_{+}$, let $A_1(T^i_{eq}(y))$ and $A_2(T^i_{eq}(y))$ be
given by (\ref{c.frequent_stability}). Then
\[
\nu(A_1(T^i_{eq}(y))),\nu(A_2(T^i_{eq}(y)))>\kappa \ .
\] Hence, if
\[
\nu\{\ y^{\prime}\in B_{\eta }(y)\cap
Y_0:d(T^{i}\varphi(y^{\prime}),T^{i}\rho(y^{\prime}))\leq\varepsilon/8\}>
\nu(B_{\eta}(y))-\kappa/2
\]
 we obtain
\[
A_{j}(T^i_{eq}(y)) \cap\{\ y^{\prime}\in B_{\eta}(y)\cap Y_0:d(T^{i}\varphi(y^{\prime}),T^{i}%
\rho(y^{\prime}))\leq\varepsilon/8\}\neq\emptyset,%
\]
for every $j\in\left\{ 1,2\right\}$. The fact that this happens for a set of $i$
of full density and we also have $d(\rho (A_{1}(T^i_{eq}(y))),\rho
(A_{2})(T^i_{eq}(y))) >\varepsilon /4$ yields
\[
\overline{D}\left\{  i\in\mathbb{N}\mid\exists y_{1}(i),y_{2}(i)\in B_{\eta
}(y):d(\varphi T_{eq}^{i}(y_{1}(i)),\varphi T_{eq}^{i}(y_{2}(i)))\geq\varepsilon/4\right\}
\ =\ 1\ .
\]
Since $\pi_{eq}$ is a factor and $\pi_{eq}\circ\varphi(y^{\prime})=y^{\prime}$ for all $y^{\prime}\in
B_{\eta}(y)$, then
\[
\overline{D}\left\{  i\in\mathbb{N}\mid\exists y_{1}(i),y_{2}(i)\in B_{\eta
}(y):d(T^{i}\varphi(y_{1}(i)),T^{i}\varphi(y_{2}(i)))\geq\varepsilon/4\right\}
\ =\ 1\ .
\]
Consequently, as $\varphi(y_{1}(i)),\varphi(y_{2}(i))\in B_{\delta}(x)$, we
obtain
\[
\overline{D}\left\{  i\in\mathbb{N}\mid\diam(T^{i}B_{\delta}(x))\geq
\varepsilon/4\right\}  \ =\ 1\ .
\]
As $x\in X$ and $\delta>0$ were arbitrary, we conclude that there are no
frequently stable points of $(X,T)$. This completes the proof.
\end{proof}

Sometimes when $\pi_{eq}:(X,T)\rightarrow(X_{eq},T_{eq})$ is almost 1-1 it is said that $(X,T)$ is \emph{almost automorphic}.
The fact that a minimal t.d.s. is mean equicontinuous if and only if $\pi_{eq}$
is isomorphic (\cite{LTY,DG}), together with the previous theorem, implies the
following corollary.

\begin{cor}
Let $(X,T)$ be a minimal t.d.s. Then $\pi_{eq}$ is isomorphic and almost 1-1  if
and only if $(X,T)$ is mean equicontinuous and frequently stable.
\end{cor}
\begin{ques} Does there exist a minimal t.d.s. that is frequently stable but
not almost automorphic (obviously non-mean equicontinuous)?
\end{ques}

We end the section with the following remark

\begin{rem}In \cite{DG}, a minimal mean equicontinuous t.d.s. where $\pi_{eq}$ is not almost 1-1
was constructed. Our result indicates that any such example cannot be frequently
stable and $\pi_{eq}$ has to be inifite to one at every point.
\end{rem}


\section{Characterization of diam-mean equicontinuity}

In this section we first give basic properties of diam-mean equicontinuous
systems, and then we show the main result of the section, i.e. that minimal
system are diam-mean equicontinuous if and only if their maximal equicontinuous
factor is regular.

\subsection{The basic properties of diam-mean equicontinuous systems}

Diam-mean equicontinuity was introduced in \cite{G}.

\begin{defn}
\label{weakeq}Let $(X,T)$ be a t.d.s. We say $x\in X$ is a \textbf{diam-mean
equicontinuity point} if for every $\varepsilon>0$ there exists $\delta>0$ such
that%
\[
\lim\sup_{N\rightarrow\infty}\frac{1}{N}\sum_{i=1}^{N}\diam(T^{i}B_{\delta
}(x))<\varepsilon.
\]
We say $(X,T)$ is \textbf{diam-mean equicontinuous }if every $x\in X$ is a
diam-mean equicontinuity point. We say $(X,T)$ is \textbf{almost diam-mean
equicontinuous }if the set of diam-mean equicontinuity points is residual.
\end{defn}
A t.d.s. is diam-mean equicontinuous if and only if for every $\varepsilon>0$
and $x\in X$ there exists $\delta>0$ such that
$\lim\sup_{N\rightarrow\infty}\frac{1}{N}\sum_{i=1}^{N}\diam(T^{i}B_{\delta}(x))<\varepsilon$.
It is not difficult to check that a transitive t.d.s. is almost diam-mean
equicontinuous if and only if there exists a diam-mean equicontinuity point and
that a minimal almost diam-mean equicontinuous system is always diam-mean
equicontinuous.

\begin{defn}
	 Let $(X,T)$ be a t.d.s. We say a set $U\subset X$ is $\varepsilon-$\textbf{stable in the mean
	}if
	\begin{equation}\label{felipe}
	\sup_{n\in\N} \{\frac{1}{N}\sum_{i=1}^{N}\diam(T^{i}U)\}<\varepsilon.
\end{equation}
	
\end{defn}
\begin{lem}
	Let $(X,T)$ be a t.d.s. 	Then $(X,T)$ is diam mean equicontinuous
	if and only if for each $\varepsilon>0$ there is $\delta>0$ such that
	for each $x\in X$, $B_{\delta }(x)$ is $\varepsilon$-stable in the mean.
\end{lem}	
\begin{proof}	
	It is clear that the late condition implies diam mean equicontinuity. Now assume that $(X,T)$ is
	diam mean equicontinuous. For each $\ep>0$ there is $\delta>0$ such that
	for each $x\in X$
	
	\begin{equation}\label{e-ye}
	\limsup_{n\to\infty}\frac{1}{n}\sum_{i=0}^{n-1} \diam(T^{i}B_{\delta}(x))<\ep.
	\end{equation}
	
	Assume the contrary that (\ref{felipe}) does not hold. Then there are $x_i\in X$, $\delta_i\rightarrow 0$,
	$\ep_0>0$ and $N_i\rightarrow \infty$ such that
	\begin{equation*}\label{equivalence}
		\frac{1}{N_i}\sum_{j=1}^{N_i}\diam(T^{j}B_{\delta_i }(x_i))\ge \varepsilon_0,\ \forall i\in \N.
	\end{equation*}
	
	Without loss of generality assume that $x_i\rightarrow x$ and $\delta=\delta(\ep_0)$. When
	$i$ is large we have that $B_{\delta_i }(x_i))\subset B_{\delta}(x)$; a contradiction since (\ref{e-ye})
	holds.
\end{proof}	




Another property of the diam-mean equicontinuity t.d.s. is the following.
\begin{lem}
\label{equival}Let $(X,T)$ be a t.d.s. Then $x\in X$ is a diam-mean
equicontinuity point if and only if for every $\eta>0$ there exists
$\delta>0$ such that
\begin{equation*}
\overline{D}\left\{  i\in\mathbb{Z}_{+}:\diam(T^{i}B_{\delta}(x))>\eta
\right\}  <\eta.
\end{equation*}

\end{lem}

\begin{proof}
We assume without loss of generality that the diameter of $X$ is bounded by $1.$

($\Rightarrow$) Let $x\in X$ be a diam-mean equicontinuity point. Assume that there exists
$\eta>0$ such that for every $\delta>0$ $\ $we have that
\[
\overline{D}\left\{  j\in\mathbb{Z}_{+}:\diam(T^{i}B_{\delta}(x))>\eta
\right\}  \geq\eta.
\]

Let $\varepsilon=\eta^{2}$.  Since $x$ is a diam-mean equicontinuity point we
may choose $\delta\in (0,\varepsilon)$ such that
\[
\lim\sup\frac{1}{n}\sum_{j=1}^{n}\diam(T^{i}B_{\delta}(x))<\varepsilon.
\]
At the same time, we have that
\begin{align*}
& \lim\sup\frac{1}{n}\sum_{j=1}^{n}\diam(T^{i}B_{\delta}(x))\\ &
  \geq\eta\overline{D}\left\{ j\in\mathbb{Z}_{+}:\diam(T^{i}%
  B_{\delta}(x))>\eta\right\} \\ & \geq\eta^{2} \ = \ \varepsilon,
\end{align*}
a contradiction.\medskip

($\Leftarrow$)
Now assume that for every $\eta>0$ there exists $\delta>0$ such that
\[
\overline{D}\left\{  i\in\mathbb{Z}_{+}:\diam(T^{i}B_{\delta}(x))>\eta
\right\}  <\eta.
\]
Given $\varepsilon>0$, let $\eta=\varepsilon/2$ and choose $\delta>0$ such that the previous
inequality holds. Then
\begin{align*}
&  \lim\sup\frac{1}{n}\sum_{j=1}^{n}\diam(T^{i}B_{\delta}(x))\\
&  \leq\overline{D}\left\{  j\in\mathbb{Z}:\diam(T^{i}B_{\delta}%
(x))>\eta\right\}  +\overline{D}\left\{  j\in\mathbb{Z}_{+}:\diam(T^{i}%
B_{\delta}(x))\leq\eta\right\}\cdot \eta\\
&  \leq2\eta \ = \ \varepsilon.
\end{align*}

Hence $x$ is a diam-mean equicontinuity point.
\end{proof}
In summary we have the following.
\begin{prop}
Let $(X,T)$ be a t.d.s. the following are equivalent:

\begin{enumerate}
	\item $(X,T)$ is diam-mean equicontinuous
	
	\item For every $\varepsilon>0$ there exists $\delta>0$ such that for every $x\in X$
	\[
	\sup_{n\in\mathbb{N}}\frac{1}{n}\sum_{j=1}^{n}\diam(T^{j}B_{\delta
	}(x))<\varepsilon.
	\]

	
	\item For every $\varepsilon>0$ there exists $\delta>0$ such that for every $x\in X$
	\[
	\overline{D}\left\{  j\in\mathbb{Z}_{+}:\diam(T^{j}B_{\delta}(x))>\varepsilon
	\right\}  \leq\varepsilon.
	\]
	
\end{enumerate}

\end{prop}

\subsection{Banach diam-mean equicontinuity and regularity}
Banach mean equicontinuity was introduced in \cite{LTY} and has been studied in \cite{QZ,DG,FGL} (on the last two papers under the name Weyl equicontinuity). In this paper we introduce the diam version.
\begin{defn}
\label{banachweak}Let $(X,T)$ be a t.d.s. We say $x\in X$ is a \textbf{Banach
diam-mean equicontinuity point} if for every $\varepsilon>0$ there exists
$\delta>0$ such that%
\[
\lim\sup_{N-M\rightarrow\infty}\frac{1}{N-M}\sum_{i=M+1}^{N}\diam(T^{i}%
B_{\delta}(x))<\varepsilon.
\]
We say $(X,T)$ is \textbf{Banach diam-mean equicontinuous }if every $x\in X$
is a diam-mean equicontinuity point. We say $(X,T)$ is \textbf{almost Banach
diam-mean equicontinuous }if the set of Banach diam-mean equicontinuity points
is residual.
\end{defn}
A t.d.s. is Banach diam-mean equicontinuous if and only
if for every $\varepsilon>0$ and $x\in X$ there exists $\delta>0$ such that
$\lim\sup_{N-M\rightarrow\infty}\frac{1}{N-M}\sum_{i=M+1}^{N}\diam(T^{i}%
B_{\delta}(x))<\varepsilon$.
A transitive t.d.s. is almost Banach diam-mean equicontinuous if and only if there exists a
Banach diam-mean equicontinuity point and a minimal almost Banach diam-mean equicontinuous system is always Banach diam-mean equicontinuous.

Every Banach diam-mean equicontinuity point is a diam-mean equicontinuity point but the converse
does not hold (see Section 5).
The proof of the following lemma is similar to the proof of Lemma
\ref{equival}.

\begin{lem}
Let $(X,T)$ t.d.s. Then $x\in X$ is a Banach diam-mean equicontinuity point if
and only if for every $\varepsilon>0$ there exists $\delta>0$ such that
\[
BD^{\ast}(B)\left\{  i\in\mathbb{Z}_{+}:\diam(T^{i}B_{\delta}(x))>\varepsilon
\right\}  <\varepsilon.
\]

\end{lem}


\begin{defn}
	We say $\pi_{eq}$ is \textbf{regular (or almost surely 1-1)} if
	 \[
	  \nu_{eq} (\left\{ y\in X_{eq}:\left\vert \pi_{eq} ^{-1}(y)\right\vert =1\right\} )=1.
	  \]
\end{defn}
The next result can be considered a measurable version of (i) implies (ii) in Theorem~\ref{t.frequent_stability}.

	\begin{prop}
		\label{regular-2} Let $(X,T)$ be a minimal t.d.s. and suppose $\pi
		:X\rightarrow X_{eq}$ is regular. Then $(X,T)$ is Banach diam-mean
		equicontinuous.
		
	\end{prop}
	
	\begin{proof}
		Without loss of generality we may assume $T_{eq}$ is an isometry.
		Let
		\[
		F=\{y\in X_{eq}:\pi_{eq}^{-1}(y)=\{x_{y}\}\}.
		\]
		Since $\pi_{eq}$ is regular we have that $\nu_{eq}(F)=1$.
		Let $\varepsilon>0$.
		There is a compact $F'\subset F$ such that $\nu_{eq  }(F')>1-\varepsilon$.
		
		For any $y\in F'$ there is $\delta_{y}>0$ such that $\pi_{eq}
                ^{-1}(B_{2\delta_{y}}(y))\subset B_{\varepsilon}(x_{y})$.
                Consider the open cover $\{B_{\delta_{y}}(y):y\in F'\}$ of
                $F'$. Since $F'$ is compact there are $n\in\mathbb{N}$ and
                $y_{1},\ldots,y_{n}\in F'$ such that
                $\cup_{i=1}^{n}B_{\delta_{y_{i}}}(y_{i})\supset F'$. Let
		
		$$\eta=\min\{\delta_{y_{i}}:1\le i\le n\}\ \text{and}
                \ G_\varepsilon=\cup_{i=1}^{n}B_{\delta_{y_{i}}}(y_{i}).$$
		
		It is clear that $G_\varepsilon$ is open and
                $\nu_{eq}(G_\varepsilon)>1-\varepsilon$. Moreover, for all $y\in
                G_\varepsilon$, there is $y_i$ such that $y\in
                B_{\delta_{y_{i}}}(y_{i})$. This implies that $B_\eta(y)\subset
                B_{\delta_{y_{i}}}(y) \subset B_{2\delta_{y_{i}}}(y_{i})$.  By
                uniform ergodicity, there exists $L>0$ such that
		\[
		\frac{1}{K}\sum_{i=j}^{K+j-1} 1_{G_\varepsilon}\circ T_{eq}^i(y) \geq 1-\varepsilon
		\]
		for all $j\in\N$, $K\geq L$ and $y\in X_{eq}$.\\ Now, choose
                $\delta>0$ such that if $U\subset X$ and $\diam(U)<\delta$ then
                $\diam(\pi_{eq}(U))<\eta/2$. Let $U$ be a non-empty open subset
                of $X$ with $\diam(U)<\delta$. Since $T_{eq}$ is an isometry we
                have that $\diam(T_{eq}^{i}\pi(U))<\eta/2$ for every
                $i\in\N$. Let $y\in\pi_{eq}(U)$. For every $i\in\N$ such that
                $T_{eq}^{i}y\in G_{\varepsilon}$ we have that
                $T_{eq}^{i}\pi(U)\subset B_{\eta}(T_{eq}^{i}y)\subset
                B_{2\delta_{y_{l}}}(y_{l})$ for some $1\leq l\leq n$. Thus,
		\[
		T^{i}(U)\subset\pi_{eq}^{-1}\pi_{eq}(T^{i}U)=\pi_{eq}^{-1}T_{eq}^{i}\pi_{eq}(U)\subset\pi_{eq}^{-1}%
                B_{2\delta_{y_{l}}}(y_{l})\subset B_{\varepsilon}(x_{y_{l}}).
		\]
		
		We conclude that when $K\ge L$
		\begin{align*}
		\frac{1}{K}\sum_{i=j}^{K+j-1}\diam(T^{i}U)&=\frac{1}{K}\sum_{\substack{i=j\\T_{eq}^{i}y\in
                    G_{\varepsilon}}}^{K+j-1}\diam(T^{i}U)
                +\frac{1}{K}\sum_{\substack{i=j\\T_{eq}^{i}y\not\in
                  G_{\varepsilon}}}^{K+j-1}\diam(T^{i}U) \\ &\le
                2\varepsilon+\diam(X) \left(\frac{1}{K}\sum_{i=j}^{K+j-1}
                1_{X\setminus G_\varepsilon}\circ T_{eq}^i(y)\right)\\ & <
                2(1+\diam(X))\varepsilon,
		\end{align*}
		finishing the proof.
		 	\end{proof}

In \cite{G} it was shown that every minimal null system is diam-mean
equicontinuous. We can now obtain a stronger result using Proposition
\ref{regular-2} and the recent result that the m.e.f. $\pi_{eq}$, of every
minimal tame t.d.s. is regular \cite{FGJO}.

\begin{cor}
Every minimal tame t.d.s. is Banach diam-mean equicontinuous.
\end{cor}

\begin{prop}
\label{regular-1} Let $(X,T)$ be a diam-mean equicontinuous minimal t.d.s.
Then $\pi_{eq}$ is regular.
\end{prop}

\begin{proof}
We have that $(X,T)$ is mean equicontinuous and frequently stable hence $\pi_{eq}:(X,T)\rightarrow(X_{eq},T_{eq})$ is almost 1-1.
Assume that $\pi_{eq}$ is not regular.

Then
\[
\nu_{eq}(\{y\in X_{eq}:\diam(\pi_{eq}^{-1}(y))>0\})=1.
\]
This implies that there is $\varepsilon>0$ such that $\delta:=\nu_{eq}
(A_{\varepsilon})>0$, where
\[
A_{\varepsilon}=\{y\in X_{eq}:\diam(\pi_{eq}^{-1}(y))>\varepsilon\}.
\]

By Birkhoff's Ergodic Theorem there is $y\in A_{\varepsilon}$ such that
\[
\frac{1}{N}\big|\{1\leq i\leq N:T_{eq}^{i}y\in A_{\varepsilon}\}\big|=\frac{1}%
{N}\sum_{i=1}^{N}1_{A_{\varepsilon}}(T_{eq}^{i}y)\rightarrow\nu_{eq}(A_{\varepsilon})>0.
\]

Now let $U$ be an non-empty open subset of $X$. Since $\pi_{eq}(U)$ contains an open
non-empty subset of $X_{eq}$, we know that there is $y_{0}\in X_{eq}$ such that
$\pi_{eq}^{-1}(y_{0})\subset U$ is a singleton, since $\pi_{eq}$ is almost 1-1. By
minimality, there is a sequence $(n_{j})_{j\in\N}$ such that
$T_{eq}^{n_{j}}y\rightarrow y_{0}$. This implies that there is $m=m(U)$
such that $T^{m}(\pi_{eq}^{-1}(y))\subset U$. It follows that
\[
\lim_{N\to\infty} \frac{1}{N}\left|\{1\leq n\leq
N:\diam(T^{n}U)>\varepsilon\}\right| \geq \lim_{N\to\infty} \frac{1}{N}\sum_{i=1}^{N}
1_{A_\varepsilon}(T_{eq}^i(T^{m}_{eq}y)) = \nu_{eq}(A_\varepsilon) = \delta,
\]
contradicting the diam-mean equicontinuity of $(X,T)$ since $\delta$ is
independent of $U$.
\end{proof}
Using much stronger hypothesis, the previous result was obtained in Theorem 54 of \cite{G}.

Combining the above two propositions we have the following main result of the section.

\begin{thm}\label{mainresult4}
Let $(X,T)$ be a minimal t.d.s. The following statements are equivalent:

\begin{enumerate}
\item $(X,T)$ is diam-mean equicontinuous.

\item $(X,T)$ is Banach diam-mean equicontinuous.

\item $\pi_{eq}:X\rightarrow X_{eq}$ is regular.

\end{enumerate}
\end{thm}

A t.d.s. (even non-minimal) is mean equicontinuous if and only if it is Banach mean equicontinuous \cite{QZ,FGL}. So we ask.

\begin{ques} Do any of the equivalences of Theorem \ref{mainresult4} hold for non-minimal systems?
\end{ques}
In the following section we will see that locally (1) and (2) are not equivalent.


\section{Diam-mean sensitivity and Almost diam-mean equicontinuity}

In this section first we present the counter part of diam-mean equicontinuity, i.e.
diam-mean sensitivity. Then we investigate the local version of diam-mean equicontinuity,
i.e. almost diam-mean equicontinuity. The main result of this section is the construction of an almost diam-mean equicontinuous systems with positive topological entropy.

\subsection{Diam-mean sensitivity}

\begin{defn}
A t.d.s. $(X,T)$ is \textbf{diam-mean sensitive} if there exists $\varepsilon>0 $
such that for every non-empty open set $U$ we have
\[
\overline{D}\left\{  i\in\mathbb{Z}_{+}:diam(T^{i}U)>\varepsilon\right\}
>\varepsilon.
\]

\end{defn}

The following result was proved in \cite{G}.

\begin{prop} \label{p.diam_mean_sensitive}
A minimal t.d.s. is either diam mean equicontinuous or diam-mean sensitive. A
transitive t.d.s. is either almost diam mean equicontinuous or diam-mean
sensitive.
\end{prop}

\begin{defn}
A t.d.s. $(X,T)$ is \textbf{Banach diam-mean sensitive} if there exists
$\varepsilon>0$ such that for every open set $U$ we have
\[
BD^{\ast}\left\{  i\in\mathbb{Z}_{+}:diam(T^{i}U)>\varepsilon\right\}
>\varepsilon.
\]

\end{defn}

The proof of the following proposition is similar to the proof of
Proposition~\ref{p.diam_mean_sensitive} in \cite{G}.

\begin{prop}
A minimal t.d.s. is either Banach diam mean equicontinuous or Banach diam-mean
sensitive. A transitive t.d.s. is either almost Banach diam mean equicontinuous
or Banach diam-mean sensitive.
\end{prop}

\subsection{Almost diam-mean equicontinuity}

If $(X,T)$ is minimal then $x\in X$ is a diam-mean equicontinuity point if and only if it is a Banach diam-mean equicontinuity point. We will see this is not true for transitve systems. \\
Almost Banach mean equicontinuous systems always have zero topological entropy but
transitive almost mean equicontinuous systems may have positive topological
entropy \cite{LTY} (for a P-system example see \cite{GLZ}). In this section we will see that even transitive almost diam-mean equicontinuous systems may have positive topological entropy.



We will construct symbolic dynamical systems. Let $\Sigma_{k}=\{0,1,%
\ldots,k-1\}^{\mathbb{N}}$ with the product topology. A metric inducing the
topology is given by $d(x,y)=0$ if $x=y$, and $d(x,y)=\frac{1}{i}$ if $x\not
=y$ and $i=\min\{i:x_{i}\not =y_{i}\}$ when $x=x_{1}x_{2}\ldots$ and $%
y=y_{1}y_{2}\ldots$. For $n\in\mathbb{N}$, we call $A\in\{0,1,\ldots,k-1%
\}^{n}$ a \textbf{finite word of length $n$} and denote $|A|=n$. For two
words $A=x_{1}\ldots x_{n}$ and $B=y_{1}\ldots y_{m}$ define $AB=x_{1}\ldots
x_{n}y_{1}\ldots y_{m}$. For a words $A$, let $[A]$ be the collection of $%
x\in\Sigma_{k}$ starting from $A$. For $x\in\Sigma_{k}$ and $i<j$, $x_{[i,j]}
$ stands for the finite words $x_{i}x_{i+1}\ldots x_{j}$.

The \textbf{shift map} $\sigma:\Sigma_{k}\rightarrow\Sigma_{k}$ is defined
by the condition that $\sigma(x)_{n}=x_{n+1}$ for $n\in\mathbb{N}$. It is
clear that $\sigma$ is a continuous surjection. The dynamical system $%
(\Sigma _{k},\sigma)$ is called the \textbf{full shift}. If $X$ is
non-empty, closed, and $\sigma$-invariant (i.e. $\sigma(X)\subset X$), then
the dynamical system $(X,\sigma)$ is called a \textbf{subshift}.


\begin{exmp}
Let $y=(y_{1},y_{2},\ldots )\in \{2,3\}^{\mathbb{N}}$ 
and $\mathcal{K}=\{k_{n}\}$ a sequence of positive integers. We will
construct a subshift $X_{y}^{\mathcal{K}}\subset \Sigma _{4}=\{0,1,2,3\}^{\N}
$.

\medskip To do this we recursively define $A_{n}$. Let $B_{n}=y_{1}\ldots
y_{n}$ , and $A_{1}=11.$ Assume that $A_{n}$ is defined. We set

\[
A_{n+1}=A_{n}0^{k_{n}}B_{n}A_{n}.
\]

Let $x=(x_{1},x_{2},\ldots )=\lim_{n\rightarrow \infty }A_{n}$. Then $x$ is
equal to
\[
A_{1}{0}^{k_{1}}B_{1}A_{1}{0}
^{k_{3}}B_{3}A_{1}0^{k_{1}}B_{1}A_{1}0^{k_{2}}B_{2}A_{1}0^{k_{1}}B_{1}A_{1}0^{k_{4}}B_{4}...
\]

Let $X_{y}^{\mathcal{K}}$ be the orbit closure of $x$. It is clear that $x$
is a recurrent point so $(X_{y}^{\mathcal{K}},\sigma )$ is transitive and
that $\overline{orb(y)}\subset X_{y}^{\mathcal{K}}$.
\end{exmp}
\begin{thm}\label{example1}
There exists $\mathcal{K}=\{k_{i}\}$ such that $(X_{y}^{\mathcal{K}},\sigma )
$ is an almost diam-mean equicontinuous t.d.s.
\end{thm}

\begin{proof}
	For $z\in X_{y}^{\mathcal{K}}$ and $i\in\N$, let $p_{i}^{z}$ be the
        smallest integer such that $z_{[p_{i}^z+1,p_{i}^z+k_{i}]}=0^{k_{i}}$.
        Set $p_{i}=p_{i}^{x}$. Note that $p_{i+1}=2p_{i}+k_{i}+i$, since it is
        easy to check that $p_i=\#A_i$. Let $\{k_{i}\}$ be a sequence defined
        inductively such that $$(2p_{i}+p_{i+1}-k_{i})/(k_{i}-p_i)\rightarrow
        0.$$ Let $\varepsilon >0$.\ We will prove that
	\[
	\limsup_{N\rightarrow \infty }\frac{1}{N}a_{N}\leq \varepsilon ,
	\]%
	where
	\[
	a_{N}=\#\{1\leq m\leq N:\text{there exists}\ z\in X_{y}^{\mathcal{K}}\cap
	\lbrack A_{1}]\ \text{such that}\ z_{m}\not=0\}.
	\]
For a given $i\in\N$, $x$ can be written as
$$x=A_i0^{k_i}B_iA_i\ 0^{k_{i+1}}B_{i+1}\ A_i0^{k_i}B_iA_i\ 0^{k_{i+2}}B_{i+2}\ A_i0^{k_i}B_iA_i\ 0^{k_{i+1}}B_{i+1}\ A_i0^{k_i}B_iA_i
\ldots.$$
Thus, for every $\sigma ^{n}(x)\in \lbrack A_{1}]$ and $i\in \mathbb{N}$ we have
  $p_{i}^{\sigma ^{n}(x)}\leq p_{i}$, since $A_1$ only appears in $A_i$ in the
  above equality, $A_i$ is followed by $0^{k_j}$ for $j\ge i$ and $\#A_i=p_i$.
  This implies that for every $z\in X_{y}^{\mathcal{K}}\cap \lbrack A_{1}]$ and
    $i\in \mathbb{N}$ we have $p_{i}^{z}\leq p_{i}$. Thus for every $i^{\prime
    }$ such that $p_{i}+1\leq i^{\prime }\leq k_{i}-p_i$ we have that
    $z_{i^{\prime }}=0$ for every $z\in X_{y}^{\mathcal{K}}\cap \lbrack
    A_{1}].$\smallskip

	Now, we assume that $(2p_{i}+p_{i+1}-k_{i})/(k_{i}-p_i)\leq \varepsilon
        $ for $i\ge i_0-1$.  For $i\ge i_0$ let $N\in \left[
          p_{i},p_{i+1}\right]$. We have that
	\begin{eqnarray*}
		\{m\in [1,N] \mid z_m\neq 0\} & \subseteq &\{1\leq m\leq
                p_{i}:\exists\ z\in X_{y}^{\mathcal{K}}\cap \lbrack
                A_{1}]\ \text{such that}\ z_{m}\not=0\} \\ & \cup & \{p_{i}+1 \leq
                  m\leq k_{i}-p_i:\exists\ z\in X_{y}^{\mathcal{K} }\cap
                    \lbrack A_{1}]\ \text{such that}\ z_{m}\not=0\} \\ &\cup&
                  \{k_{i}-p_i+1 \leq m\leq p_{i+1}:\exists\ z\in
                  X_{y}^{\mathcal{K} }\cap \lbrack A_{1}]\ \text{such
                    that}\ z_{m}\not=0\} \ .
	\end{eqnarray*}
        By the previous argument we have that
	\[
	\#\{p_{i}+1\leq m\leq k_{i}-p_{i}:\exists\ z\in X_{y}^{\mathcal{K}}\cap
	\lbrack A_{1}]\ \text{such that}\ z_{m}\not=0\}=0.
	\]
		We claim that for all $N\in [k_i-p_i+1,k_{i+1}-p_{i+1}]$ we have
                $\frac{1}{N}a_{N}\leq \varepsilon $. As this holds for all
                $i\geq i_0$ and $\varepsilon>0$ was arbitrary, this will imply
                $\lim_{N\to\infty} a_N/N=0$.\smallskip

       If $N\in [k_i-p_i+1,p_{i+1}]$, then
$$\frac{1}{N}a_{N}\leq \frac{1}{k_{i}-p_i}a_{N}\leq
       \frac{1}{k_{i}-p_i}(p_{i}+p_{i+1}-(k_{i}-p_{i}))<\ep$$ If $N\in
            [p_{i+1},k_{i+1}-p_{i+1}]$, then $ a_N=a_{p_{i+1}}$ and hence
     $a_N/N\leq a_{p_{i+1}}/p_{i+1} < \varepsilon
     $
This implies that
	\[
	\lim_{N\rightarrow \infty }\frac{1}{N}a_{N}=0.
	\]

Hence, we have $\lim_{N\to\infty} a_N/N=0$ as stated above. To conclude the
proof we need a standard argument. There is $K_{\varepsilon }\in \mathbb{N}$
such that for any $a,b\in \Sigma _{4}$ if $a[1,K_{\varepsilon
}]=b[1,K_{\varepsilon }]$ then $d(a,b)<\varepsilon$
	\begin{align*}
	\limsup_{N\rightarrow \infty }\frac{1}{N}\sum_{i=1}^{N}\diam(T^{i}[A_{1}])&
	\leq \varepsilon\cdot \limsup_{N\rightarrow \infty }\frac{1}{N}%
	\#\{1\leq i\leq N\colon (T^{i}z)_{[1,K_\ep]}=0^{K_\ep},\forall \ z\in \lbrack
	A_{1}]\} \\
	& +\limsup_{N\rightarrow \infty }\frac{1}{N}\#\{1\leq i\leq N\colon
	(T^{i}z)_{[1,K_\ep]}\neq 0^{K_\ep},\text{for some}\ z\in \lbrack A_{1}]\} \\
	& \leq \varepsilon+\limsup_{N\rightarrow \infty }\frac{1}{N}%
	K_{\varepsilon }a_{N} = \varepsilon \ .
	\end{align*}
	
	Thus $x$ is a diam-mean equicontinuity point.
\end{proof}

Furthermore if we choose $y$ so that $(\overline{orb(y)},\sigma)$ has
positive topological entropy 
then we obtain the following corollary.

\begin{cor}
	There exists a transitive almost diam-mean equicontinuous t.d.s. with
        positive topological entropy.
\end{cor}
Almost Banach diam-mean equicontinuous systems are almost Banach mean
equicontinuous and thus they always have zero topological entropy
\cite{LTY}. This yields the following.
\begin{cor}
	There exists diam-mean equicontinuity points on transitive systems that
        are not Banach diam-mean equicontinuity points.
\end{cor}

 In summary, for local properties we have the following diagram.
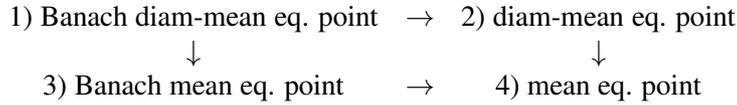
\[%
\begin{array}
[c]{ccc} \text{1) Banach diam-mean eq. point} & \rightarrow & \text{2)
diam-mean eq.  point}\\ \downarrow & & \downarrow\\ \text{3) Banach mean
  eq. point} & \rightarrow & \text{4) mean eq. point}%
\end{array}
\]
Every implication is strict. Furthermore, there is no relationship between 2)
and 3). By the same reasoning of the previous corollary we have that there exist
diam-mean equicontinuity points that are not Banach mean equicontinuity
points. On the other hand, 3) does not imply 2); consider a non diam-mean
equicontinuous system that is mean equicontinuous, and hence Banach mean
equicontinuous.
\begin{rem}
	Note that in the proof of Theorem~\ref{example1} we prove that the
        average diameter of $A_1$ is zero. This property is stronger than almost
        diam-mean equicontinuity.
\end{rem}

\section{Furstenberg's question}

Given a t.d.s. $(X,T)$ we say a point $x$ is \textbf{minimal} if
$(\overline{orb(x)},T)$ is a minimal t.d.s.  In \cite[p. 231]{F81}
Furstenberg asked if for every t.d.s. $(X,T)$ and $d\in \N$, there is $x\in X$
such that $(x,x,\ldots,x)$ is a minimal point for $T\times T^2\times \ldots
\times T^d$. We have a positive answer for the class of mean equicontinuous
systems.  Since each t.d.s. has a minimal subset, we only need to consider the
question for a minimal t.d.s.  Recall that a minimal mean equicontinuous system
is uniquely ergodic.
	
\begin{thm} Let $(X,T)$ be minimal mean equicontinuous t.d.s. There is a Borel set $X_0$
	with $\mu(X_0)=1$ (where $\mu$ is the unique measure) such that for any
        $d\in\N$ and $x\in X_0$, $(x,x,\ldots,x)$ is a minimal point for
        $T\times T^2\times \ldots \times T^d$.
\end{thm}
\begin{proof} By \cite{LTY,DG} the maximal equicontinuous factor map $\pi_{eq}$ yields an isomorphism and
$(X,T,\mu)$ has discrete spectrum, and hence is measurably distal. Let $\{V_i\}$
  be the base for the topology of $X$.

Let $X_0$ be the set such that the pointwise multiple ergodic theorem
(\cite{HSY}) holds for $x\in X_0$, each $1_{V_i}$ and each $d\in\N$, i.e.
$$\frac{1}{N}\sum_{n=0}^{N-1} 1_{V_i}(T^nx)1_{V_i}(T^{2n}x)\ldots
        1_{V_i}(T^{dn}x)$$ converges. Thus $\mu(X_0)=1$. Moreover, by
        Furstenberg the limit is positive.
	
	Now fix $d\in\N$.
Then for $x\in X_0$, and any non-empty open neighborhood $U=B_\ep(x)$ of $x$,
there is $i\in\N$ such that $x\in V_i\subset U$.  Thus,
	
\begin{equation}\label{multiple}
\liminf_{N\rightarrow \infty} \frac{1}{N}\sum_{n=0}^{N-1} 1_U(T^nx)\ldots 1_U(T^{dn}x)\ge
	\lim_{N\rightarrow \infty} \frac{1}{N}\sum_{n=0}^{N-1} 1_{V_i}(T^nx)\ldots 1_{V_i}(T^{dn}x)>0.
\end{equation}

	Let $y=\pi_{eq}(x)$ and $\pi_{eq}^{-1}y=X_y$. Since $(y,y,\ldots,y)$ is
        minimal for $T_{eq}\times T^2_{eq}\times \ldots \times T_{eq}^{d}$ there
        is a minimal point $(x_1,x_2,\ldots,x_d)\in X_y^d$ for $T\times
        T^2\times \ldots \times T^d$. We may assume that $x_1=x$. To see this,
        let $T^{n_i}x_1\rightarrow x$ and assume that $T^{2n_i}x_2\rightarrow
        x_2', \ldots, T^{dn_i}x_d\rightarrow x_d'$. Since $X_{eq}$ is
        equicontinuous, $T^{n_i}y\rightarrow y$ implies $T^{jn_i}y\rightarrow y$
        for $j=2,\ldots,d$. We conclude that $\pi_{eq}(x_j')=y$ for
        $j=2,\ldots,d$.

	Let
	$$B_1=\{n\in\N: T^nx\in U, T^{2n}x\in U\}\ \text{and}\ B_j=\{n\in\N:
        d(T^{jn}x,T^{jn}x_j)<\ep\},\ 2\le j\le d.$$
	
	Since $B_1$ has positive upper density (see (\ref{multiple})) and $B_j$
        has Banach density 1 (see \cite{LTY} by considering $T^j$), there must
        exist $n\in \cap_{j=1}^dB_j\not=\emptyset$. Thus, we have
	$$T^n(x)\in U=B_{\ep}(x)\ \text{and}\ T^{jn}x_j\in B_{2\ep}(x),\ 2\le j\le d.$$
	We conclude that $(x,x, \ldots,x)$ is in the orbit closure of $(x,x_2, \ldots,x_d)$ under $T\times T^2\times \ldots \times T^d$
and hence minimal under $T\times T^2\times \ldots \times T^d$.
	\end{proof}
 This is not the first time that an open question is answered partially for mean
 equicontinuous systems. For instance, in \cite{DK} Downarowicz and Kasjan
 proved the Sarnak conjecture in this case.


\begin{thebibliography}{99}
	\bibitem{A} J. Auslander, \emph{Minimal flows and their extensions},
	North-Holland Publishing Co., Amsterdam, 1988.
	
	\bibitem{A59} J. Auslander, \emph{Mean-$L$-stable systems}, Illinois J.
	Math., \textbf{3} (1959): 566--579.

        \bibitem{BLM} M. Baake, D. Lenz, and R.V. Moody. \emph{Characterization
          of model sets by dynamical systems}. Ergodic Theory Dynam. Systems
          \textbf{27}(2) (2007): 341--382.

	\bibitem{AY} J. Auslander and J.A. Yorke, \emph{Interval maps, factors of
		maps, and chaos}, Tohoku Mathematical Journal, \textbf{32}(2) (1980):
	177--579.
	
	\bibitem{D} T. Downarowicz, \textit{Survey of odometers and Toeplitz flows},
	Contemporary Mathematics 385 (2005): 7-38.
	
	\bibitem{DG} T. Downarowicz and E. Glasner, \textit{Isomorphic extension and
		applications}, Topological Methods in Nonlinear Analysis \textbf{48}.1 (2016): 321-338
	
	\bibitem{DK} T. Downarowicz and S. Kasjan, \textit{Odometers and Toeplitz systems revisited in the context of Sarnak's conjecture}, Studia Mathematica \textbf{229} (2015), 45-72
	
	\bibitem{EG} R. Ellis and W.H. Gottschalk, \textit{Homomorphisms of transformation groups}, Transactions of the American Mathematical Society \textbf{94} (1960): 258-271
	
	\bibitem{F} S. Fomin, \emph{On dynamical systems with a purely point spectrum}, Dokl. Akad. Nauk SSSR, \textbf{77} (1951):29-32 (in Russian).
	
	\bibitem{F81} H. Furstenberg, \emph{Poincare recurrence and number theory}, Bull. Amer. Math. Soc. \textbf{5}.3 (1981): 211-234.
	
	\bibitem{FGJO} G. Fuhrmann, E. Glasner, T. J\"{a}ger, and C. Oertel, \emph{%
		Irregular model sets and tame dynamics}, arXiv preprint:1811.06283 (2018).
	
	\bibitem{FGL} G. Fuhrmann, M. Gr\"{o}ger, and D. Lenz, \emph{The structure
		of mean equicontinuous group actions}, arXiv preprint arXiv:1812.10219
	(2018).

        \bibitem{FK} G. Fuhrmann, and D. Kwietniak. On tameness of almost
          automorphic dynamical systems for general groups, arXiv
          preprint:1902.10780 (2019).
	
	\bibitem{G} F. Garc\'{\i}a-Ramos, \emph{Weak forms of topological and
		measure theoretical equicontinuity: relationships with discrete spectrum and
		sequence entropy}, Ergodic Theory and Dynamical Systems \textbf{37}.4
	(2017): 1211-1237.
	
	\bibitem{GM} F. Garc\'{\i}a-Ramos and B. Marcus, \emph{Mean sensitive, mean
		equicontinuous and almost periodic functions for dynamical systems,}
	Discrete and Continuous Dynamical Systems - A, \textbf{39}.2 (2019): 729-746.
	
	\bibitem{GLZ} F. Garc\'{\i}a-Ramos, J. Li, and R. Zhang, \emph{When is a
		dynamical system mean sensitive?,} Ergodic Theory and Dynamical Systems \textbf{39}.6 (2019): 1608-1636.
	
	\bibitem{G06} E. Glasner, \emph{"The structure of tame minimal dynamical systems}, Ergodic Theory and Dynamical Systems, \textbf{27}.6 (2007):1819-1837.
	
	\bibitem{G18} E. Glasner, \emph{The structure of tame minimal dynamical
		systems for general groups}, Inventiones mathematicae \textbf{211}.1 (2018):
	213-244.
	
	\bibitem{goodman} T. N. T.  Goodman, \emph{Topological sequence entropy}, Proceedings of the London Mathematical Society \textbf{3}.2 (1974): 331-350.
	
	\bibitem{HV} P. Halmos, and J. Von Neumann \emph{Operator methods in
          classical mechanics II}, Annals of Mathematics (2) \textbf{43} (1942):
          332-350.

	
	\bibitem{HLTY} W. Huang, J. Li, J.P. Thouvenot, L. Xu, and X. Ye, \emph{%
		Bounded complexity, mean equicontinuity and discrete spectrum}. arXiv
	preprint:1806.02980 (2018). Ergod. Th. \& Dynam. Sys., to appear.
	
	\bibitem{HLSY} W. Huang, S. Li, S. Shao and X. Ye, \emph{Null systems and
		sequence entropy pairs}, Ergod. Th. \& Dynam. Sys., \textbf{23} (2003):
	1505--1523.

	
	\bibitem{HSY} W. Huang, S. Shao and X. Ye, \emph{Pointwise convergence of multiple ergodic averages and strictly ergodic models}.
     arXiv:1406.5930, J. d'Anal. Math., to appear.

	
	\bibitem{KL05} D. Kerr and H. Li, \emph{Dynamical entropy in Banach spaces},
	Inventiones mathematicae \textbf{162}.3 (2005): 649-686.
	
	\bibitem{KL07} D. Kerr and H. Li, \emph{Independence in topological and C*-dynamics}, Mathematische Annalen, \textbf{338}.4 (2007):
	869--926.
	
	
	\bibitem{kush} A. G. Kushnirenko, \emph{On metric invariants of entropy type}, Russian Mathematical Surveyss \textbf{22}.5 (1967): 53-61.
	
	\bibitem{L87} E. Lehrer, \emph{Topological mixing and uniquely ergodic systems}, Israel Journal of Mathematics \textbf{57}.2 (1972): 239-255.
	
	
	\bibitem{LTY} J. Li, S. Tu and X. Ye, \emph{Mean equicontinuity and mean
		sensitivity}, Ergodic Theory and Dynamical Systems \textbf{35}.8 (2015):
	2587-2612.
	
	\bibitem{QZ} J. Qiu and J. Zhao. \emph{A note on mean equicontinuity.}
	Journal of Dynamics and Differential Equations (to appear).
	
	
	\bibitem{S} B. Scarpellini, \emph{Stability properties of flows with pure point spectrum}, Journal of the London Mathematical Society, \textbf{2}.3 (1982): 451--464.
	
	\bibitem{V} J. Von Neumann \emph{Zur Operatorenmethode in der
          klassischen Mechanik}, Annals of Mathematics (2) \textbf{33} (1932): 587--642.
	
	\bibitem{YZ} X. Ye and R. Zhang, \emph{On sensitivity sets in topological
		dynamics}, Nonlinearity, \textbf{21} (2008): 1601--1620.
\end{thebibliography}
\end{document}